\documentclass{amsart}
\usepackage{amssymb}
\usepackage{amsmath}
\usepackage{amsfonts}

\newtheorem{theorem}{Theorem}
\theoremstyle{plain}

\newtheorem{corollary}{Corollary}

\newtheorem{definition}{Definition}
\newtheorem{example}{Example}

\newtheorem{proposition}{Proposition}
\newtheorem{remark}{Remark}

\numberwithin{equation}{section}
\input{tcilatex}

\begin{document}
\title[Symmetrized $p$-convexity]{Symmetrized $p$-convexity and Related Some
Integral Inequalities}
\author{\.{I}mdat \.{I}\c{s}can}
\address{Department of Mathematics, Faculty of Arts and Sciences, Giresun
University, 28200, Giresun, Turkey}
\email{imdat.iscan@giresun.edu.tr, imdati@yahoo.com}

\begin{abstract}
In this paper, the author introduces the concept of the symmetrized $p$%
-convex function, gives Hermite-Hadamard type inequalities for symmetrized $p
$-convex functions.
\end{abstract}

\maketitle

\section{Introduction}

Let real function $f$ be defined on some nonempty interval $I$ of real line $%
\mathbb{R}
$. The function $f$ is said to be convex on $I$ if inequality%
\begin{equation}
f(tx+(1-t)y)\leq tf(x)+(1-t)f(y)  \label{0-1}
\end{equation}%
holds for all $x,y\in I$ and $t\in \left[ 0,1\right] .$

In \cite{I14}, the author, gave definition Harmonically convex and concave
functions as follow.

\begin{definition}
Let $I\subset 
\mathbb{R}
\backslash \left\{ 0\right\} $ be a real interval. A function $%
f:I\rightarrow 
\mathbb{R}
$ is said to be harmonically convex, if \ 
\begin{equation}
f\left( \frac{xy}{tx+(1-t)y}\right) \leq tf(y)+(1-t)f(x)  \label{0-2}
\end{equation}%
for all $x,y\in I$ and $t\in \lbrack 0,1]$. If the inequality in (\ref{0-2})
is reversed, then $f$ is said to be harmonically concave.
\end{definition}

The following result of the Hermite-Hadamard type holds for harmonically
convex functions.

\begin{theorem}[\protect\cite{I14}]
\label{1.3}Let $f:I\subset 
\mathbb{R}
\backslash \left\{ 0\right\} \rightarrow 
\mathbb{R}
$ be a harmonically convex function and $a,b\in I$ with $a<b.$ If $f\in
L[a,b]$ then the following inequalities hold 
\begin{equation}
f\left( \frac{2ab}{a+b}\right) \leq \frac{ab}{b-a}\dint\limits_{a}^{b}\frac{%
f(x)}{x^{2}}dx\leq \frac{f(a)+f(b)}{2}.  \label{0-3}
\end{equation}%
The \ above inequalities are sharp.
\end{theorem}

In \cite{I16}, the author gave the definition of $p$-convex function as
follow:

\begin{definition}[\protect\cite{I16}]
\label{0.2}Let $I\subset \left( 0,\infty \right) $ be a real interval and $%
p\in 
\mathbb{R}
\backslash \left\{ 0\right\} .$ A function $f:I\rightarrow 
\mathbb{R}
$ is said to be a p-convex function, if%
\begin{equation}
f\left( \left[ tx^{p}+(1-t)y^{p}\right] ^{1/p}\right) \leq tf(x)+(1-t)f(y)
\label{0-4}
\end{equation}%
for all $x,y\in I$ and $t\in \lbrack 0,1]$. If the inequality in (\ref{0-4})
is reversed, then $f$ is said to be $p$-concave.
\end{definition}

According to Definition \ref{0.2}, It can be easily seen that for $p=1$ and $%
p=-1$, $p$-convexity reduces to ordinary convexity and harmonically
convexity of functions defined on $I\subset \left( 0,\infty \right) $,
respectively.

Since the condition (\ref{0-4}) can be written as%
\begin{equation*}
\left( f\circ g\right) \left( tx^{p}+(1-t)y^{p}\right) \leq t\left( f\circ
g\right) (x^{p})+(1-t)\left( f\circ g\right) (y^{p}),\ g(x)=x^{1/p},
\end{equation*}%
then we observe that $f:I\subseteq \left( 0,\infty \right) \rightarrow 
\mathbb{R}
$ is $p$-convex on $I$ if and only if $f\circ g$ is convex on $%
g^{-1}(I):=\{z^{p},z\in I\}$ .

\begin{example}
Let $f:\left( 0,\infty \right) \rightarrow 
\mathbb{R}
,\ f(x)=x^{p},p\neq 0,$ and $g:\left( 0,\infty \right) \rightarrow 
\mathbb{R}
,\ g(x)=c,~\ c\in 
\mathbb{R}
,$ then $f$ and $g$ are both $p$-convex and $p$-concave functions.
\end{example}

In \cite[Theorem 5]{FS14}, if we take $I\subset \left( 0,\infty \right) $, $%
p\in 
\mathbb{R}
\backslash \left\{ 0\right\} $ and $h(t)=t$ , then we have the following
Theorem.

\begin{theorem}
\label{2.2}Let $f:I\subseteq \left( 0,\infty \right) \rightarrow 
\mathbb{R}
$ be a $p$-convex function, $p\in 
\mathbb{R}
\backslash \left\{ 0\right\} $, and $a,b\in I$ with $a<b.$ If $f\in L[a,b]$
then we have 
\begin{equation}
f\left( \left[ \frac{a^{p}+b^{p}}{2}\right] ^{1/p}\right) \leq \frac{p}{%
b^{p}-a^{p}}\dint\limits_{a}^{b}\frac{f(x)}{x^{1-p}}dx\leq \frac{f(a)+f(b)}{2%
}.  \label{0-5}
\end{equation}
\end{theorem}

\begin{definition}[\protect\cite{KI18}]
Let $p\in 
\mathbb{R}
\backslash \left\{ 0\right\} .$ A function $w:\left[ a,b\right] \subseteq
\left( 0,\infty \right) \rightarrow 
\mathbb{R}
$ is said to be p-symmetric with respect to $\left[ \frac{a^{p}+b^{p}}{2}%
\right] ^{1/p}$ if%
\begin{equation*}
w(x)=w\left( \left[ a^{p}+b^{p}-x^{p}\right] ^{1/p}\right)
\end{equation*}%
holds for all $x\in \left[ a,b\right] .$
\end{definition}

In \cite{KI17}, Kunt and \.{I}\c{s}can gave Hermite-Hadamard-Fej\'{e}r type
inequalities for $p$-convex functions as follow:

\begin{theorem}
\label{T-0}Let $f:I\subseteq \left( 0,\infty \right) \rightarrow 
\mathbb{R}
$ be a $p$-convex function, $p\in 
\mathbb{R}
\backslash \left\{ 0\right\} ,a,b\in I$ with $a<b.$ If $f\in L[a,b]$ and $%
w:[a,b]\rightarrow 
\mathbb{R}
$ is nonnegative, integrable and $p$-symmetric with respect to $\left[ \frac{%
a^{p}+b^{p}}{2}\right] ^{1/p}$, then the following inequalities hold:%
\begin{eqnarray}
&&f\left( \left[ \frac{a^{p}+b^{p}}{2}\right] ^{1/p}\right)
\dint\limits_{a}^{b}\frac{w(x)}{x^{1-p}}dx\leq \dint\limits_{a}^{b}\frac{%
f(x)w(x)}{x^{1-p}}dx  \notag \\
&\leq &\frac{f(a)+f(b)}{2}\dint\limits_{a}^{b}\frac{w(x)}{x^{1-p}}dx.
\label{0-5-a}
\end{eqnarray}
\end{theorem}

\begin{definition}
Let $f\in L\left[ a,b\right] $. The left-sided and right-sided Hadamard
fractional integrals $J_{a+}^{\alpha }f$ and $J_{b-}^{\alpha }f$ of oder $%
\alpha >0$ with $b>a\geq 0$ are defined by

\begin{equation*}
J_{a+}^{\alpha }f(x)=\frac{1}{\Gamma (\alpha )}\dint\limits_{a}^{x}\left(
x-t\right) ^{\alpha -1}f(t)dt,\ a<x<b
\end{equation*}

and

\begin{equation*}
J_{b-}^{\alpha }f(x)=\frac{1}{\Gamma (\alpha )}\dint\limits_{x}^{b}\left(
t-x\right) ^{\alpha -1}f(t)dt,\ a<x<b
\end{equation*}%
respectively, where $\Gamma (\alpha )$ is the Gamma function defined by $%
\Gamma (\alpha )=$ $\dint\limits_{0}^{\infty }e^{-t}t^{\alpha -1}dt$ (see 
\cite{KST06}).
\end{definition}

In \cite{KI18}, the authors presented Hermite--Hadamard-Fejer inequalities
for $p$-convex functions in fractional integral forms as follows:

\begin{theorem}
\label{T-1}Let $f:I\subseteq \left( 0,\infty \right) \rightarrow 
\mathbb{R}
$ be a $p$-convex function, $p\in 
\mathbb{R}
\backslash \left\{ 0\right\} ,\alpha >0$ and $a,b\in I$ with $a<b.$ If $f\in
L[a,b]$ and $w:[a,b]\rightarrow 
\mathbb{R}
$ is nonnegative, integrable and $p$-symmetric with respect to $\left[ \frac{%
a^{p}+b^{p}}{2}\right] ^{1/p}$, then the following inequalities for
fractional integrals hold:

i.) If $p>0$%
\begin{eqnarray}
&&f\left( \left[ \frac{a^{p}+b^{p}}{2}\right] ^{1/p}\right) \left[
J_{a^{p}+}^{\alpha }\left( w\circ g\right) (b^{p})+J_{b^{p}-}^{\alpha
}\left( w\circ g\right) (a^{p})\right]  \label{0-6} \\
&\leq &\left[ J_{a^{p}+}^{\alpha }\left( fw\circ g\right)
(b^{p})+J_{b^{p}-}^{\alpha }\left( fw\circ g\right) (a^{p})\right]  \notag \\
&\leq &\frac{f(a)+f(b)}{2}\left[ J_{a^{p}+}^{\alpha }\left( w\circ g\right)
(b^{p})+J_{b^{p}-}^{\alpha }\left( w\circ g\right) (a^{p})\right] .  \notag
\end{eqnarray}%
with $g(x)=x^{1/p},x\in \left[ a^{p},b^{p}\right] $.

ii.) If $p>0$%
\begin{eqnarray*}
&&f\left( \left[ \frac{a^{p}+b^{p}}{2}\right] ^{1/p}\right) \left[
J_{b^{p}+}^{\alpha }\left( w\circ g\right) (a^{p})+J_{a^{p}-}^{\alpha
}\left( w\circ g\right) (b^{p})\right] \\
&\leq &\left[ J_{b^{p}+}^{\alpha }\left( fw\circ g\right)
(a^{p})+J_{a^{p}-}^{\alpha }\left( fw\circ g\right) (b^{p})\right] \\
&\leq &\frac{f(a)+f(b)}{2}\left[ J_{b^{p}+}^{\alpha }\left( w\circ g\right)
(a^{p})+J_{a^{p}-}^{\alpha }\left( w\circ g\right) (b^{p})\right] .
\end{eqnarray*}%
with $g(x)=x^{1/p},x\in \left[ b^{p},a^{p}\right] $.
\end{theorem}

For a function $f:[a,b]\mathbb{\rightarrow R}$ we consider the symmetrical
transform of $f$ on the interval , denoted by $\overset{\smile }{f}_{[a,b]}$
or simply $\overset{\smile }{f}$, when the interval $[a,b]$ is implicit,
which

is defined by 
\begin{equation*}
\overset{\smile }{f}(x):=\frac{1}{2}\left[ f(x)+f\left( a+b-x\right) \right]
,~x\in \left[ a,b\right] .
\end{equation*}%
The anti symmetrical transform of $f$ on the interval $[a,b]$ is denoted by $%
\overset{\sim }{f}_{[a,b]}$or simply $\overset{\sim }{f}$ as defined by%
\begin{equation*}
\overset{\sim }{f}(x):=\frac{1}{2}\left[ f(x)-f\left( a+b-x\right) \right]
,~x\in \left[ a,b\right] .
\end{equation*}%
It is obvious that for any function $f$ we have $\overset{\smile }{f}+%
\overset{\sim }{f}=f$.

If $f$ is convex on $[a,b]$, then $\overset{\smile }{f}$ is also convex on $%
[a,b].$ But, when $\overset{\smile }{f}$ is onvex on $[a,b]$, $f$ may not be
convex on $[a,b]$ ( \cite{D16}).

In \cite{D16}, Dragomir introduced symmetrized convexity concept as follow:

\begin{definition}
A function $f:[a,b]\mathbb{\rightarrow R}$ is said to be symmetrized convex
(concave)on $[a,b]$ if symmetrical transform $\overset{\smile }{f}$ is
convex (concave) on $[a,b]$.
\end{definition}

Dragomir extends the Hermite-Hadamard inequality to the class of symmetrized
convex functions as follow:

\begin{theorem}[\protect\cite{D16}]
\label{T1-1}Assume that $f:[a,b]\mathbb{\rightarrow R}$ is symmetrized
convex on the interval $[a,b]$, then we have the Hermite-Hadamard
inequalities%
\begin{equation}
f\left( \frac{a+b}{2}\right) \leq \frac{1}{b-a}\int_{a}^{b}f(x)dx\leq \frac{%
f(a)+f(b)}{2}.  \label{1-2}
\end{equation}
\end{theorem}

\begin{theorem}[\protect\cite{D16}]
\label{T1-2}Assume that $f:[a,b]\mathbb{\rightarrow R}$ is symmetrized
convex on the interval $[a,b]$. Then for any $x\in \lbrack a,b]$ we have the
bounds%
\begin{equation}
f\left( \frac{a+b}{2}\right) \leq \overset{\smile }{f}\left( x\right) =\frac{%
1}{2}\left[ f(x)+f(a+b-x)\right] \leq \frac{f(a)+f(b)}{2}  \label{1-3}
\end{equation}
\end{theorem}

\begin{corollary}
\label{C1-1}If $f:[a,b]\mathbb{\rightarrow R}$ is symmetrized convex on the
interval $[a,b]$ and $w:[a,b]\rightarrow \lbrack 0,\infty )$is integrable on 
$[a,b]$, then%
\begin{equation*}
f\left( \frac{a+b}{2}\right) \int_{a}^{b}w(x)dx\leq
\int_{a}^{b}f(x)w(x)dx\leq \frac{f(a)+f(b)}{2}\int_{a}^{b}w(x)dx.
\end{equation*}
\end{corollary}

\begin{theorem}[\protect\cite{D16}]
\label{T1-3}Assume that $f:[a,b]\subseteq \left( 0,\infty \right)
\rightarrow 
\mathbb{R}
$ is symmetrized convex on the interval $[a,b].$ Then for any $x,y\in
\lbrack a,b]$ with $x\neq y$ we have the Hermite-Hadamard inequalities%
\begin{eqnarray}
&&\frac{1}{2}\left[ f\left( \frac{x+y}{2}\right) +f\left( a+b-\frac{x+y}{2}%
\right) \right]  \label{1-4} \\
&\leq &\frac{1}{2\left( y-x\right) }\left[ \int_{x}^{y}f(t)dt+%
\int_{a+b-y}^{a+b-x}f(t)dt\right]  \notag \\
&\leq &\frac{1}{4}\left[ f(x)+f(y)+f\left( a+b-x\right) +f\left(
a+b-y\right) \right] .  \notag
\end{eqnarray}
\end{theorem}

For a function $f:[a,b]\subseteq \mathbb{R\diagdown }\left\{ 0\right\} 
\mathbb{\rightarrow }\mathbb{%
\mathbb{C}
}$, we consider the symmetrical transform of $f$ on the interval , denoted
by $\overset{\smile }{Hf}_{[a,b]}$ or simply $\overset{\smile }{Hf}$, when
the interval $[a,b]$ as defined by 
\begin{equation*}
\overset{\smile }{Hf}(x):=\frac{1}{2}\left[ f(x)+f\left( \frac{abx}{\left(
a+b\right) x-ab}\right) \right] ,~x\in \left[ a,b\right] .
\end{equation*}

\begin{definition}[\protect\cite{WABH17}]
A function $f:I\subseteq \mathbb{R\diagdown }\left\{ 0\right\} \mathbb{%
\rightarrow R}$ is said to be symmetrized harmonic convex (concave)on $[a,b]$
if $\overset{\smile }{Hf}$ is harmonic convex (concave) on $I$.
\end{definition}

The similars of above results given for the class of symmetrized convex
functions, in \cite{WABH17} it has been obtained by Wu et al. for the class
of symmetrized harmonic convex functions as follow:

\begin{theorem}[\protect\cite{WABH17}]
\label{T1-4}Assume that $f:[a,b]\subseteq \mathbb{R\diagdown }\left\{
0\right\} \mathbb{\rightarrow R}$ is symmetrized harmonic convex and
integrable on the interval $[a,b]$. Then we have the Hermite-Hadamard type 
\.{I}\c{s}can inequalities%
\begin{equation}
f\left( \frac{2ab}{a+b}\right) \leq \frac{ab}{b-a}\dint\limits_{a}^{b}\frac{%
f(x)}{x^{2}}dx\leq \frac{f(a)+f(b)}{2}.  \label{1-5}
\end{equation}
\end{theorem}

\begin{theorem}[\protect\cite{WABH17}]
\label{T1-5}Assume that $f:[a,b]\subseteq \mathbb{R\diagdown }\left\{
0\right\} \mathbb{\rightarrow R}$ is symmetrized harmonic convex on the
interval $[a,b]$. Then for any $x\in \lbrack a,b]$ we have the bounds%
\begin{equation}
f\left( \frac{2ab}{a+b}\right) \leq \overset{\smile }{Hf}\left( x\right) =%
\frac{1}{2}\left[ f(x)+f\left( \frac{abx}{\left( a+b\right) x-ab}\right) %
\right] \leq \frac{f(a)+f(b)}{2}  \label{1-6}
\end{equation}
\end{theorem}

\begin{theorem}[\protect\cite{WABH17}]
\label{T1-6}Assume that $f:[a,b]\subseteq \mathbb{R\diagdown }\left\{
0\right\} \rightarrow 
\mathbb{R}
$ is symmetrized harmonic convex on the interval $[a,b].$ Then for any $%
x,y\in \lbrack a,b]$ with $x\neq y$ we have the Hermite-Hadamard inequalities%
\begin{eqnarray}
&&\frac{1}{2}\left[ f\left( \frac{2xy}{x+y}\right) +f\left( \frac{2abxy}{%
2xy\left( a+b\right) -ab(x+y)}\right) \right]  \label{1-7} \\
&\leq &\frac{xy}{2\left( y-x\right) }\left[ \int_{x}^{y}\frac{f(t)}{t^{2}}%
dt+\int_{\frac{aby}{\left( a+b\right) y-ab}}^{\frac{abx}{\left( a+b\right)
x-ab}}\frac{f(t)}{t^{2}}dt\right]  \notag \\
&\leq &\frac{1}{4}\left[ f(x)+f(y)+f\left( \frac{abx}{\left( a+b\right) x-ab}%
\right) +f\left( \frac{aby}{\left( a+b\right) y-ab}\right) \right] .  \notag
\end{eqnarray}
\end{theorem}

Motivated by the above results, in this paper we introduces the concept of
the symmetrized $p$-convex function and establish some Hermite-Hadamard type
inequalities. Some examples of interest are provided as well. \ 

\section{Symmetrized $p$-Convexity}

For a function $f:[a,b]\subseteq \left( 0,\infty \right) \mathbb{\rightarrow
R}$ we consider the $p$-symmetrical transform of $f$ on the interval ,
denoted by $P_{\left( f;p\right) ,[a,b]}$ or simply $P_{\left( f;p\right) }$%
, when the interval $[a,b]$ is implicit, which

is defined by 
\begin{equation*}
P_{\left( f;p\right) }(x):=\frac{1}{2}\left[ f(x)+f\left( \left[
a^{p}+b^{p}-x^{p}\right] ^{1/p}\right) \right] ,~x\in \left[ a,b\right] .
\end{equation*}%
The anti $p$-symmetrical transform of $f$ on the interval $[a,b]$ is denoted
by $AP_{\left( f;p\right) ,[a,b]}$ or simply $AP_{\left( f;p\right) }$ as
defined by%
\begin{equation*}
AP_{\left( f;p\right) }(x):=\frac{1}{2}\left[ f(x)-f\left( \left[
a^{p}+b^{p}-x^{p}\right] ^{1/p}\right) \right] ,~x\in \left[ a,b\right] .
\end{equation*}%
It is obvious that for any function $f$ we have $P_{\left( f;p\right)
}+AP_{\left( f;p\right) }=f$. Also it is seen that $P_{\left( f;1\right)
}(x)=\frac{1}{2}\left[ f(x)+f\left( a+b-x\right) \right] =\overset{\smile }{f%
}(x)$ and $P_{\left( f;-1\right) }(x)=\frac{1}{2}\left[ f(x)+f\left( \frac{%
abx}{\left( a+b\right) x-ab}\right) \right] =\overset{\smile }{Hf}(x).$

If $f$ is $p$-convex on $[a,b]$, then $P_{\left( f;p\right) }$ is also $p$%
-convex on $[a,b].$ Indeed, for any $x,y\in \lbrack a,b]$ and $t\in \left[
0,1\right] $ we have%
\begin{eqnarray*}
&&P_{\left( f;p\right) }(\left[ tx^{p}+(1-t)y^{p}\right] ^{1/p}) \\
&=&\frac{1}{2}\left[ f(\left[ tx^{p}+(1-t)y^{p}\right] ^{1/p})+f\left( \left[
a^{p}+b^{p}-tx^{p}-(1-t)y^{p}\right] ^{1/p}\right) \right] \\
&=&\frac{1}{2}\left[ f(\left[ tx^{p}+(1-t)y^{p}\right] ^{1/p})+f\left( \left[
t\left( a^{p}+b^{p}-x^{p}\right) +(1-t)\left( a^{p}+b^{p}-y^{p}\right) %
\right] ^{1/p}\right) \right] \\
&\leq &t\frac{1}{2}\left[ f(x)+f\left( \left[ a^{p}+b^{p}-x^{p}\right]
^{1/p}\right) \right] +(1-t)\frac{1}{2}\left[ f(y)+f\left( \left[
a^{p}+b^{p}-y^{p}\right] ^{1/p}\right) \right] \\
&=&tP_{\left( f;p\right) }(x)+(1-t)P_{\left( f;p\right) }(y).
\end{eqnarray*}

\begin{remark}
If $P_{\left( f;p\right) }$ is $p$-convex on $[a,b]$ for a function $%
f:[a,b]\subseteq \left( 0,\infty \right) \mathbb{\rightarrow R}$, then the
function $f$ is nor necessary $p$-convex on $[a,b]$. For example, let $p=-1$%
, Consider the function $f(x)=-\ln x,x\in \left( 0,\infty \right) $. The
function $f$ is not $-1$-convex (or harmonically convex), but $P_{\left(
f;-1\right) }$ is $-1$-convex \cite{WABH17}.
\end{remark}

\begin{definition}
A function $f:[a,b]\subseteq \left( 0,\infty \right) \mathbb{\rightarrow R}$
is said to be symmetrized $p$-convex ($p$-concave)on $[a,b]$ if $p$%
-symmetrical transform $P_{\left( f;p\right) }$ is $p$-convex ($p$-concave)
on $[a,b]$.
\end{definition}

\begin{example}
Let $a,b\in 
\mathbb{R}
$ with $0<a<b$ and $\alpha \geq 2$. Then the function $f:\left[ a,b\right]
\rightarrow 
\mathbb{R}
,~f(x)=\left( x^{p}-a^{p}\right) ^{\alpha -1},$ $p>0,$ (or $f(x)=\left(
a^{p}-x^{p}\right) ^{\alpha -1},$ $p<0$) is $p$-convex on $\left[ a,b\right] 
$. Indeed, for any $u,v\in \left[ a,b\right] $ and $t\in \left[ 0,1\right] $
by convexity of the function $g(\varsigma )=\varsigma ^{\alpha -1}$,$\zeta
\geq 0$, we have%
\begin{eqnarray*}
f(\left[ tu^{p}+(1-t)v^{p}\right] ^{1/p}) &=&\left(
tu^{p}+(1-t)v^{p}-a^{p}\right) ^{\alpha -1} \\
&=&\left( t\left[ u^{p}-a^{p}\right] +(1-t)\left[ v^{p}-a^{p}\right] \right)
^{\alpha -1} \\
&\leq &t\left( u^{p}-a^{p}\right) ^{\alpha -1}+(1-t)\left(
v^{p}-a^{p}\right) ^{\alpha -1} \\
&=&tf(u)+(1-t)f(v).
\end{eqnarray*}%
Thus $P_{\left( f;p\right) }$ is also $p$-convex on $[a,b]$. Therefore $f$
is symmetrized $p$-convex function.
\end{example}

\begin{example}
Let $\alpha \geq 2$. Then the function $f:\left[ a,b\right] \subseteq \left(
0,\infty \right) \rightarrow 
\mathbb{R}
,~f(x)=\left( b^{p}-x^{p}\right) ^{\alpha -1},$ $p>0,$ (or $f(x)=\left(
x^{p}-b^{p}\right) ^{\alpha -1},$ $p<0$) is $p$-convex on $\left[ a,b\right]
.$Therefore $f$ is symmetrized $p$-convex function.
\end{example}

\begin{example}
Let $\alpha \geq 2$. Then the function $f:\left[ a,b\right] \subseteq \left(
0,\infty \right) \rightarrow 
\mathbb{R}
,~f(x)=\left( x^{p}-a^{p}\right) ^{\alpha -1}+\left( b^{p}-x^{p}\right)
^{\alpha -1},\ p>0,$ (or $f(x)=\left( a^{p}-x^{p}\right) ^{\alpha -1}+\left(
x^{p}-b^{p}\right) ^{\alpha -1},$ $p<0$) is symmetrized $p$-convex function.
\end{example}

Now if $PC[a,b]$ is the class of $p$-convex functions defined on I and $%
SPC[a,b]$ is the class of symmetrized $p$-convex functions on $[a,b]$ then%
\begin{equation*}
PC[a,b]\subsetneqq SPC[a,b].
\end{equation*}%
Also, if $[c,d]\subset \lbrack a,b]$ and $f\in SPC[a,b]$, then this does not
imply in general that $f\in SPC[c,d]$.

\begin{proposition}
\label{P2-1}Let $f:[a,b]\subseteq \left( 0,\infty \right) \mathbb{%
\rightarrow R}$ be a function and $g(x)=x^{1/p},x>0,p\neq 0$. $f$ is
symmetrized $p$-convex on the interval $[a,b]$ if and only if $f\circ g$ is
symmetrized convex on the interval $g^{-1}\left( [a,b]\right) =[a^{p},b^{p}]$
(or $[b^{p},a^{p}]$).
\end{proposition}

\begin{proof}
Let $f$ be a symmetrized $p$-convex function on the interval $[a,b]$. If we
take arbitrary $x,y\in g^{-1}\left( [a,b]\right) $, then there exist $u,v\in
\lbrack a,b]$ such that $x=u^{p}$ and $y=g(v)=v^{p}$ 
\begin{eqnarray}
&&\overset{\smile }{\left( f\circ g\right) }(tx+(1-t)y)  \label{1-1a} \\
&=&\frac{1}{2}\left[ \left( f\circ g\right) (tx+(1-t)y)+\left( f\circ
g\right) \left( a^{p}+b^{p}-tx-(1-t)y\right) \right]  \notag \\
&=&\frac{1}{2}\left[ \left( f\circ g\right) (tu^{p}+(1-t)v^{p})+\left(
f\circ g\right) \left( a^{p}+b^{p}-\left[ tu^{p}+(1-t)v^{p}\right] \right) %
\right]  \notag \\
&=&P_{\left( f;p\right) }(\left[ tu^{p}+(1-t)v^{p}\right] ^{1/p}).  \notag
\end{eqnarray}%
Since $f$ is a symmetrized $p$-convex function on the interval $[a,b],$ we
have%
\begin{equation}
P_{\left( f;p\right) }(\left[ tu^{p}+(1-t)v^{p}\right] ^{1/p})\leq
tP_{\left( f;p\right) }(u)+(1-t)P_{\left( f;p\right) }(v)  \label{1-1b}
\end{equation}%
\begin{eqnarray*}
&=&t\frac{1}{2}\left[ f(u)+f\left( \left[ a^{p}+b^{p}-u^{p}\right]
^{1/p}\right) \right] +(1-t)\frac{1}{2}\left[ f(v)+f\left( \left[
a^{p}+b^{p}-v^{p}\right] ^{1/p}\right) \right] \\
&=&t\frac{1}{2}\left[ \left( f\circ g\right) (x)+\left( f\circ g\right)
\left( a^{p}+b^{p}-x\right) \right] +(1-t)\frac{1}{2}\left[ \left( f\circ
g\right) (y)+\left( f\circ g\right) \left( a^{p}+b^{p}-v^{p}\right) \right]
\\
&=&\overset{\smile }{t\left( f\circ g\right) }(x)+(1-t)\overset{\smile }{%
\left( f\circ g\right) }(y)
\end{eqnarray*}%
By (\ref{1-1a}) and (\ref{1-1b}), we obtain that $f\circ g$ is symmetrized
convex on the interval $[a^{p},b^{p}]$ (or $[b^{p},a^{p}]$).

Conversely, if $f\circ g$ is symmetrized convex on the interval $%
[a^{p},b^{p}]$ (or $[b^{p},a^{p}]$) then it is easily seen that $f$ is
symmetrized $p$-convex on the interval $[a,b]$ by a similar procedure. The
details are omitted.
\end{proof}

\begin{theorem}
\label{T2-1}If $f:[a,b]\subseteq \left( 0,\infty \right) \mathbb{\rightarrow
R}$ is symmetrized $p$-convex on the interval $[a,b]$, then we have the
Hermite-Hadamard inequalities%
\begin{equation}
f\left( \left[ \frac{a^{p}+b^{p}}{2}\right] ^{1/p}\right) \leq \frac{p}{%
b^{p}-a^{p}}\int_{a}^{b}\frac{f(x)}{x^{1-p}}dx\leq \frac{f(a)+f(b)}{2}.
\label{2-1}
\end{equation}
\end{theorem}

\begin{proof}
Since $f:[a,b]\subseteq \left( 0,\infty \right) \mathbb{\rightarrow R}$ is
symmetrized $p$-convex on the interval $[a,b]$, then by writing the
Hermite-Hadamard inequality for the function $P_{\left( f;p\right) }(x)$ we
have%
\begin{equation}
P_{\left( f;p\right) }\left( \left[ \frac{a^{p}+b^{p}}{2}\right]
^{1/p}\right) \leq \frac{p}{b^{p}-a^{p}}\int_{a}^{b}\frac{P_{\left(
f;p\right) }(x)}{x^{1-p}}dx\leq \frac{P_{\left( f;p\right) }(a)+P_{\left(
f;p\right) }(b)}{2},  \label{2-1a}
\end{equation}%
where, it is easily seen that%
\begin{equation*}
P_{\left( f;p\right) }\left( \left[ \frac{a^{p}+b^{p}}{2}\right]
^{1/p}\right) =f\left( \left[ \frac{a^{p}+b^{p}}{2}\right] ^{1/p}\right) ,~%
\frac{P_{\left( f;p\right) }(a)+P_{\left( f;p\right) }(b)}{2}=\frac{f(a)+f(b)%
}{2},
\end{equation*}%
and%
\begin{equation*}
\frac{p}{b^{p}-a^{p}}\int_{a}^{b}\frac{P_{\left( f;p\right) }(x)}{x^{1-p}}dx=%
\frac{p}{b^{p}-a^{p}}\int_{a}^{b}\frac{f(x)}{x^{1-p}}dx
\end{equation*}%
Then by (\ref{2-1a}) we get required inequalities.
\end{proof}

\begin{remark}
In Theorem \ref{T2-1},

i.) if we choose $p=1$, then the inequalities (\ref{2-2}) reduces to the
inequalities (\ref{1-2}) in Theorem (\ref{T1-1}).

ii.) if we choose $p=-1$, then the inequalities (\ref{2-2}) reduces to the
inequalities (\ref{1-5}) in Theorem (\ref{T1-4}).
\end{remark}

\begin{remark}
By helping Theorem \ref{T1-1} and Proposition \ref{P2-1}, the proof of
Theorem \ref{T2-1} can also be given as follows :

Since $f:[a,b]\subseteq \left( 0,\infty \right) \mathbb{\rightarrow R}$ is
symmetrized $p$-convex on the interval $[a,b]$, $f\circ g$ is symmetrized
convex on the interval $[a^{p},b^{p}]$ (or $[b^{p},a^{p}]$) with $%
g(x)=x^{1/p},x>0,p\neq 0.$ So, by Theorem \ref{T1-1} we have%
\begin{equation*}
\left( f\circ g\right) \left( \frac{a^{p}+b^{p}}{2}\right) \leq \frac{1}{%
b^{p}-a^{p}}\int_{a^{p}}^{b^{p}}\left( f\circ g\right) (x)dx\leq \frac{%
\left( f\circ g\right) (a^{p})+\left( f\circ g\right) (b^{p})}{2},
\end{equation*}%
i.e.%
\begin{equation*}
f\left( \left[ \frac{a^{p}+b^{p}}{2}\right] ^{1/p}\right) \leq \frac{p}{%
b^{p}-a^{p}}\int_{a}^{b}\frac{f(x)}{x^{1-p}}dx\leq \frac{f(a)+f(b)}{2}.
\end{equation*}
\end{remark}

\begin{theorem}
\label{T2-2}If $f:[a,b]\subseteq \left( 0,\infty \right) \mathbb{\rightarrow
R}$ is symmetrized $p$-convex on the interval $[a,b]$. Then for any $x\in
\lbrack a,b]$ we have the bounds%
\begin{equation}
f\left( \left[ \frac{a^{p}+b^{p}}{2}\right] ^{1/p}\right) \leq P_{\left(
f;p\right) }\left( x\right) \leq \frac{f(a)+f(b)}{2}.  \label{2-2}
\end{equation}
\end{theorem}

\begin{proof}
Since$P_{\left( f;p\right) }$ is $p$-convex on $[a,b]$ then for any $x\in
\lbrack a,b]$ we have%
\begin{equation*}
f\left( \left[ \frac{a^{p}+b^{p}}{2}\right] ^{1/p}\right) =P_{\left(
f;p\right) }\left( \left[ \frac{a^{p}+b^{p}}{2}\right] ^{1/p}\right) \leq 
\frac{P_{\left( f;p\right) }(x)+P_{\left( f;p\right) }(\left[
a^{p}+b^{p}-x^{p}\right] ^{1/p})}{2}=P_{\left( f;p\right) }\left( x\right) .
\end{equation*}%
This give us the first inequality in (\ref{2-2}).

Also, for any $x\in \lbrack a,b]$ there exist a number $t_{0}\in \lbrack
0,1] $ such that $x=\left[ t_{0}a^{p}+(1-t_{0})b^{p}\right] ^{1/p}$. By the $%
p$-convexity of $P_{\left( f;p\right) }$ we have%
\begin{eqnarray*}
P_{\left( f;p\right) }\left( x\right) &\leq &t_{0}P_{\left( f;p\right)
}\left( a\right) +(1-t_{0})P_{\left( f;p\right) }\left( b\right) \\
&=&P_{\left( f;p\right) }\left( a\right) =\frac{f(a)+f(b)}{2}
\end{eqnarray*}

which gives the second inequality in (\ref{2-2}).
\end{proof}

\begin{remark}
In Theorem \ref{T2-2},

i.) if we choose $p=1$, then the inequalities (\ref{2-2}) reduces to the
inequalities (\ref{1-3}) in Theorem (\ref{T1-2}).

ii.) if we choose $p=-1$, then the inequalities (\ref{2-2}) reduces to the
inequalities (\ref{1-6}) in Theorem (\ref{T1-5}).
\end{remark}

\begin{remark}
By helping Theorem \ref{T1-2} and Proposition \ref{P2-1}, the proof of
Theorem \ref{T2-2} can also be given as follows :

Since $f:[a,b]\subseteq \left( 0,\infty \right) \mathbb{\rightarrow R}$ is
symmetrized $p$-convex on the interval $[a,b]$, $f\circ g$ is symmetrized
convex on the interval $[a^{p},b^{p}]$ with $g(x)=x^{1/p},x>0,p\neq 0$. So,
by Theorem \ref{T1-2} we have%
\begin{equation*}
\left( f\circ g\right) \left( \frac{a^{p}+b^{p}}{2}\right) \leq \overset{%
\smile }{\left( f\circ g\right) }(x^{p})\leq \frac{\left( f\circ g\right)
(a^{p})+\left( f\circ g\right) (b^{p})}{2},
\end{equation*}%
i.e.%
\begin{equation*}
f\left( \left[ \frac{a^{p}+b^{p}}{2}\right] ^{1/p}\right) \leq P_{\left(
f;p\right) }\left( x\right) \leq \frac{f(a)+f(b)}{2}
\end{equation*}%
for any $x\in \lbrack a,b]$.
\end{remark}

\begin{remark}
If $f:[a,b]\subseteq \left( 0,\infty \right) \mathbb{\rightarrow R}$ is
symmetrized $p$-convex on the interval $[a,b]$, then we have the bounds%
\begin{equation*}
\underset{x\in \lbrack a,b]}{\inf }P_{\left( f;p\right) }\left( x\right)
=f\left( \left[ \frac{a^{p}+b^{p}}{2}\right] ^{1/p}\right)
\end{equation*}%
and%
\begin{equation*}
\underset{x\in \lbrack a,b]}{\sup }P_{\left( f;p\right) }\left( x\right) =%
\frac{f(a)+f(b)}{2}.
\end{equation*}
\end{remark}

\begin{corollary}
\label{C2-1}If $f:[a,b]\subseteq \left( 0,\infty \right) \mathbb{\rightarrow
R}$ is symmetrized $p$-convex on the interval $[a,b]$ and $%
w:[a,b]\rightarrow \lbrack 0,\infty )$is integrable on $[a,b]$, then%
\begin{equation}
f\left( \left[ \frac{a^{p}+b^{p}}{2}\right] ^{1/p}\right) \int_{a}^{b}\frac{%
w(x)}{x^{1-p}}dx\leq \int_{a}^{b}\frac{w(x)P_{\left( f;p\right) }\left(
x\right) }{x^{1-p}}dx\leq \frac{f(a)+f(b)}{2}\int_{a}^{b}\frac{w(x)}{x^{1-p}}%
dx.  \label{2-3-1}
\end{equation}%
Moreover, if $w$ is $p$-symmetric with respect to $\left[ \frac{a^{p}+b^{p}}{%
2}\right] ^{1/p}$on $[a,b]$, i.e. $w(x)=w(\left[ a^{p}+b^{p}-x^{p}\right]
^{1/p})$ for all $x\in \lbrack a,b]$, then%
\begin{equation}
f\left( \left[ \frac{a^{p}+b^{p}}{2}\right] ^{1/p}\right) \int_{a}^{b}\frac{%
w(x)}{x^{1-p}}dx\leq \int_{a}^{b}\frac{w(x)f\left( x\right) }{x^{1-p}}dx\leq 
\frac{f(a)+f(b)}{2}\int_{a}^{b}\frac{w(x)}{x^{1-p}}dx.  \label{2-3-2}
\end{equation}
\end{corollary}

\begin{proof}
The inequality (\ref{2-3-1}) follows by (\ref{2-2}) multiplying by $%
w(x)/x^{1-p}\geq 0$ and integrating over $x$ on $[a,b]$.

By changing the variable, we have

\begin{equation*}
\int_{a}^{b}\frac{w(x)f\left( \left[ a^{p}+b^{p}-x^{p}\right] ^{1/p}\right) 
}{x^{1-p}}dx=\int_{a}^{b}\frac{w(\left[ a^{p}+b^{p}-x^{p}\right]
^{1/p})f\left( x\right) }{x^{1-p}}dx.
\end{equation*}%
Since $w$ is $p$-symmetric with respect to $\left[ \frac{a^{p}+b^{p}}{2}%
\right] ^{1/p}$, then%
\begin{equation*}
\int_{a}^{b}\frac{w(\left[ a^{p}+b^{p}-x^{p}\right] ^{1/p})f\left( x\right) 
}{x^{1-p}}dx=\int_{a}^{b}\frac{w(x)f\left( x\right) }{x^{1-p}}dx.
\end{equation*}%
Thus%
\begin{eqnarray*}
\int_{a}^{b}\frac{w(x)P_{\left( f;p\right) }\left( x\right) }{x^{1-p}}dx &=&%
\frac{1}{2}\left[ \int_{a}^{b}\frac{w(x)f\left( x\right) }{x^{1-p}}%
dx+\int_{a}^{b}\frac{w(x)f\left( \left[ a^{p}+b^{p}-x^{p}\right]
^{1/p}\right) }{x^{1-p}}dx\right] \\
&=&\int_{a}^{b}\frac{w(x)f\left( x\right) }{x^{1-p}}dx
\end{eqnarray*}%
and by (\ref{2-3-1}) we get (\ref{2-3-2}).
\end{proof}

\begin{remark}
The inequality (\ref{2-3-2}) is known as weighted generalization of
Hermite-Hadamard inequality for $p$-convex functions (it is also given in
Theorem \ref{T-0}). It has been shown now that this inequality remains valid
for the larger class of symmetrized $p$-convex functions $f$ on the interval 
$[a,b]$.
\end{remark}

\begin{remark}
We note that by helping Corollary \ref{C1-1} and Proposition \ref{P2-1}, the
proof of Corollary \ref{C2-1} can also be given. The details is omitted.
\end{remark}

\begin{remark}
Let $a,b,\alpha \in 
\mathbb{R}
$ with $0<a<b$ and $\alpha \geq 2$. Then the function $f:\left[ a,b\right]
\rightarrow 
\mathbb{R}
,~f(x)=\left( x^{p}-a^{p}\right) ^{\alpha -1},$ $p>0,$ is symmetrized $p$%
-convex on $\left[ a,b\right] $

i.) If we consider the function 
\begin{equation*}
f(x)=\left( x^{p}-a^{p}\right) ^{\alpha -1}
\end{equation*}%
which is symmetrized $p$-convex on $\left[ a,b\right] $ in the inequality (%
\ref{2-3-1}), then we have%
\begin{equation*}
\frac{1}{2^{\alpha -1}}\int_{a}^{b}\frac{w(x)}{x^{1-p}}dx\leq \frac{\Gamma
(\alpha )}{2p\left( b^{p}-a^{p}\right) ^{\alpha -1}}\left[
J_{a^{p}+}^{\alpha }\left( w\circ g\right) (b^{p})+J_{b^{p}-}^{\alpha
}\left( w\circ g\right) (a^{p})\right] \leq \frac{1}{2}\int_{a}^{b}\frac{w(x)%
}{x^{1-p}}dx
\end{equation*}%
for any $w:[a,b]\rightarrow \lbrack 0,\infty )$ is integrable on $[a,b]$
with $g(x)=x^{1/p},x\in \left[ a^{p},b^{p}\right] .$

ii.) If we consider the function 
\begin{equation*}
w(x)=\left( x^{p}-a^{p}\right) ^{\alpha -1}+\left( b^{p}-x^{p}\right)
^{\alpha -1}
\end{equation*}%
which is $p$-symmetric with respect to $\left[ \frac{a^{p}+b^{p}}{2}\right]
^{1/p}$ in the inequality (\ref{2-3-2}), then we have the following
inequalities 
\begin{equation*}
f\left( \left[ \frac{a^{p}+b^{p}}{2}\right] ^{1/p}\right) \leq \frac{\Gamma
(\alpha +1)}{2\left( b^{p}-a^{p}\right) ^{\alpha }}\left[ J_{a^{p}+}^{\alpha
}\left( f\circ g\right) (b^{p})+J_{b^{p}-}^{\alpha }\left( f\circ g\right)
(a^{p})\right] \leq \frac{f(a)+f(b)}{2},
\end{equation*}
where $g(x)=x^{1/p},x\in \left[ a^{p},b^{p}\right] .$

iii.) Let $\varphi $ be $p$-symmetric with respect to $\left[ \frac{%
a^{p}+b^{p}}{2}\right] ^{1/p}.$ If we consider the function 
\begin{equation*}
w(x)=\left[ \left( x^{p}-a^{p}\right) ^{\alpha -1}+\left( b^{p}-x^{p}\right)
^{\alpha -1}\right] \varphi (x)
\end{equation*}%
which is $p$-symmetric with respect to $\left[ \frac{a^{p}+b^{p}}{2}\right]
^{1/p}$ in the inequality (\ref{2-3-2}), then we have the following
inequalities%
\begin{eqnarray*}
&&f\left( \left[ \frac{a^{p}+b^{p}}{2}\right] ^{1/p}\right) \left[
J_{a^{p}+}^{\alpha }\left( \varphi \circ g\right) (b^{p})+J_{b^{p}-}^{\alpha
}\left( \varphi \circ g\right) (a^{p})\right] \\
&\leq &\left[ J_{a^{p}+}^{\alpha }\left( f\varphi \circ g\right)
(b^{p})+J_{b^{p}-}^{\alpha }\left( f\varphi \circ g\right) (a^{p})\right] \\
&\leq &\frac{f(a)+f(b)}{2}\left[ J_{a^{p}+}^{\alpha }\left( \varphi \circ
g\right) (b^{p})+J_{b^{p}-}^{\alpha }\left( \varphi \circ g\right) (a^{p})%
\right]
\end{eqnarray*}%
which are the same of inequalities in (\ref{0-6}). Where $g(x)=x^{1/p},x\in %
\left[ a^{p},b^{p}\right] .$
\end{remark}

\begin{theorem}
\label{T2-3}Assume that $f:[a,b]\subseteq \left( 0,\infty \right)
\rightarrow 
\mathbb{R}
$ is symmetrized $p$-convex on the interval $[a,b]$ with $p\in 
\mathbb{R}
\diagdown \left\{ 0\right\} .$ Then for any $x,y\in \lbrack a,b]$ with $%
x\neq y$ we have the Hermite-Hadamard inequalities%
\begin{eqnarray}
&&\frac{1}{2}\left[ f(\left[ \frac{x^{p}+y^{p}}{2}\right] ^{1/p})+f\left( %
\left[ a^{p}+b^{p}-\frac{x^{p}+y^{p}}{2}\right] ^{1/p}\right) \right]
\label{2-4} \\
&\leq &\frac{p}{2\left( y^{p}-x^{p}\right) }\left[ \int_{x}^{y}\frac{f(t)}{%
t^{1-p}}dt+\int_{\left[ a^{p}+b^{p}-y^{p}\right] ^{1/p}}^{\left[
a^{p}+b^{p}-x^{p}\right] ^{1/p}}\frac{f(t)}{t^{1-p}}dt\right]  \notag \\
&\leq &\frac{1}{4}\left[ f(x)+f(y)+f\left( \left[ a^{p}+b^{p}-x^{p}\right]
^{1/p}\right) +f\left( \left[ a^{p}+b^{p}-y^{p}\right] ^{1/p}\right) \right]
.  \notag
\end{eqnarray}
\end{theorem}

\begin{proof}
Since $P_{\left( f;p\right) ,[a,b]}$ is $p$-convex on $[a,b]$, then $%
P_{\left( f;p\right) ,[a,b]}$ is also $p$-convex on any subinterval $[x,y]$
(or $[y,x]$) where $x,y\in \lbrack a,b].$

By Hermite-Hadamard inequalities for convex functions we have%
\begin{equation}
P_{\left( f;p\right) ,[a,b]}\left( \left[ \frac{x^{p}+y^{p}}{2}\right]
^{1/p}\right) \leq \frac{p}{y^{p}-x^{p}}\int_{x}^{y}\frac{P_{\left(
f;p\right) ,[a,b]}(t)}{t^{1-p}}dt\leq \frac{P_{\left( f;p\right)
,[a,b]}(x)+P_{\left( f;p\right) ,[a,b]}(y)}{2}  \label{2-4-a}
\end{equation}%
for any $x,y\in \lbrack a,b]$ with $x\neq y$.

By definition of $P_{\left( f;p\right) }$, we have%
\begin{equation*}
P_{\left( f;p\right) ,[a,b]}\left( \left[ \frac{x^{p}+y^{p}}{2}\right]
^{1/p}\right) =\frac{1}{2}\left[ f(\left[ \frac{x^{p}+y^{p}}{2}\right]
^{1/p})+f\left( \left[ a^{p}+b^{p}-\frac{x^{p}+y^{p}}{2}\right]
^{1/p}\right) \right] ,
\end{equation*}%
\begin{eqnarray*}
\int_{x}^{y}\frac{P_{\left( f;p\right) ,[a,b]}(t)}{t^{1-p}}dt &=&\frac{1}{2}%
\int_{x}^{y}\frac{1}{t^{1-p}}\left[ f(t)+f\left( \left[ a^{p}+b^{p}-t^{p}%
\right] ^{1/p}\right) \right] dt \\
&=&\frac{1}{2}\int_{x}^{y}\frac{f(t)}{t^{1-p}}dt+\frac{1}{2}\int_{x}^{y}%
\frac{f\left( \left[ a^{p}+b^{p}-t^{p}\right] ^{1/p}\right) }{t^{1-p}}dt \\
&=&\frac{1}{2}\int_{x}^{y}\frac{f(t)}{t^{1-p}}dt+\frac{1}{2}\int_{\left[
a^{p}+b^{p}-y^{p}\right] ^{1/p}}^{\left[ a^{p}+b^{p}-x^{p}\right] ^{1/p}}%
\frac{f(t)}{t^{1-p}}dt
\end{eqnarray*}%
and%
\begin{eqnarray*}
&&\frac{P_{\left( f;p\right) ,[a,b]}(x)+P_{\left( f;p\right) ,[a,b]}(y)}{2}
\\
&=&\frac{1}{4}\left[ f(x)+f(y)+f\left( \left[ a^{p}+b^{p}-x^{p}\right]
^{1/p}\right) +f\left( \left[ a^{p}+b^{p}-y^{p}\right] ^{1/p}\right) \right]
.
\end{eqnarray*}%
Thus by (\ref{2-4-a}) we obtain the desired result (\ref{2-4}).
\end{proof}

\begin{remark}
We note that by helping Theorem \ref{T1-3} and Proposition \ref{P2-1}, the
proof of Theorem \ref{T2-3} can also be given. The details is omitted.
\end{remark}

\begin{remark}
If we take $x=a$ and $y=b$ in (\ref{2-4}), then we get (\ref{2-1}). If, for
a given $x\in \lbrack a,b]$, we take $y=\left[ a^{p}+b^{p}-x^{p}\right]
^{1/p}$, then from (\ref{2-4}) we get%
\begin{equation}
f\left( \left[ \frac{a^{p}+b^{p}}{2}\right] ^{1/p}\right) \leq \frac{p}{%
a^{p}+b^{p}-2x^{p}}\int_{x}^{\left[ a^{p}+b^{p}-x^{p}\right] ^{1/p}}\frac{%
f(t)}{t^{1-p}}dt\leq \frac{1}{2}\left[ f(x)+f\left( \left[ a^{p}+b^{p}-x^{p}%
\right] ^{1/p}\right) \right] ,  \label{2-4b}
\end{equation}%
where $x\neq \left[ \frac{a^{p}+b^{p}}{2}\right] ^{1/p}$, provided that $%
f:[a,b]\subseteq \left( 0,\infty \right) \rightarrow 
\mathbb{R}
$ is symmetrized $p$-convex on the interval $[a,b]$.

Multiplying the inequalities (\ref{2-4b}) by $\frac{1}{x^{1-p}}$, then
integrating the resulting inequality over x we get the following refinement
of the first part of (\ref{2-1})%
\begin{eqnarray*}
&&f\left( \left[ \frac{a^{p}+b^{p}}{2}\right] ^{1/p}\right) \\
&\leq &\frac{p^{2}}{\left( b^{p}-a^{p}\right) }\int_{a}^{b}\left[ \frac{1}{%
x^{1-p}\left( a^{p}+b^{p}-2x^{p}\right) }\int_{x}^{\left[ a^{p}+b^{p}-x^{p}%
\right] ^{1/p}}\frac{f(t)}{t^{1-p}}dt\right] dx \\
&\leq &\frac{p}{b^{p}-a^{p}}\int_{a}^{b}\frac{f(x)}{x^{1-p}}dx,
\end{eqnarray*}%
provided that $f:[a,b]\subseteq \left( 0,\infty \right) \rightarrow 
\mathbb{R}
$ is symmetrized $p$-convex on the interval $[a,b]$.
\end{remark}

When the function is $p$-convex, we have the following inequalities as well:

\begin{remark}
If $f:[a,b]\subseteq \left( 0,\infty \right) \rightarrow 
\mathbb{R}
$ is $p$-convex, then from (\ref{2-4}) we have the inequalities%
\begin{eqnarray}
&&f\left( \left[ \frac{a^{p}+b^{p}}{2}\right] ^{1/p}\right)  \label{2-4c} \\
&\leq &\frac{1}{2}\left[ f\left( \left[ \frac{x^{p}+y^{p}}{2}\right]
^{1/p}\right) +f\left( \left[ a^{p}+b^{p}-\frac{x^{p}+y^{p}}{2}\right]
^{1/p}\right) \right]  \notag \\
&\leq &\frac{p}{2\left( y^{p}-x^{p}\right) }\left[ \int_{x}^{y}\frac{f(t)}{%
t^{1-p}}dt+\int_{\left[ a^{p}+b^{p}-y^{p}\right] ^{1/p}}^{\left[
a^{p}+b^{p}-x^{p}\right] ^{1/p}}\frac{f(t)}{t^{1-p}}dt\right]  \notag \\
&\leq &\frac{1}{4}\left[ f(x)+f(y)+f\left( \left[ a^{p}+b^{p}-x^{p}\right]
^{1/p}\right) +f\left( \left[ a^{p}+b^{p}-y^{p}\right] ^{1/p}\right) \right]
,  \notag
\end{eqnarray}%
for any $x,y\in \lbrack a,b]$ with $x\neq y.$

If we multiply the inequalities(\ref{2-4c}) by $\frac{1}{\left( xy\right)
^{1-p}}$ and integrate (\ref{2-4c}) over $(x,y)$ on the square $[a,b]^{2}$
and divide by $\frac{p^{2}}{\left( b^{p}-a^{p}\right) ^{2}}$ , then we get
the following refinement of the first Hermite-Hadamard inequality for $p$%
-convex functions%
\begin{eqnarray*}
&&f\left( \left[ \frac{a^{p}+b^{p}}{2}\right] ^{1/p}\right) \\
&\leq &\frac{p^{2}}{2\left( b^{p}-a^{p}\right) ^{2}}\left[
\int_{a}^{b}\int_{a}^{b}\frac{f\left( \left[ \frac{x^{p}+y^{p}}{2}\right]
^{1/p}\right) }{\left( xy\right) ^{1-p}}dxdy+\int_{a}^{b}\int_{a}^{b}\frac{%
f\left( \left[ a^{p}+b^{p}-\frac{x^{p}+y^{p}}{2}\right] ^{1/p}\right) }{%
\left( xy\right) ^{1-p}}dxdy\right] \\
&\leq &\frac{p^{2}}{2\left( b^{p}-a^{p}\right) ^{2}}\int_{a}^{b}\int_{a}^{b}%
\frac{1}{\left( xy\right) ^{1-p}\left( y^{p}-x^{p}\right) }\left[
\int_{x}^{y}\frac{f(t)}{t^{1-p}}dt+\int_{\left[ a^{p}+b^{p}-y^{p}\right]
^{1/p}}^{\left[ a^{p}+b^{p}-x^{p}\right] ^{1/p}}\frac{f(t)}{t^{1-p}}dt\right]
dxdy \\
&\leq &\frac{p}{\left( b^{p}-a^{p}\right) }\int_{a}^{b}\frac{f(x)}{x^{1-p}}%
dx.
\end{eqnarray*}
\end{remark}

\begin{remark}
In Theorem \ref{T2-3},

i.) if we choose $p=1$, then the inequalities (\ref{2-4}) reduces to the
inequalities (\ref{1-4}) in Theorem (\ref{T1-3}).

ii.) if we choose $p=-1$, then the inequalities (\ref{2-4}) reduces to the
inequalities (\ref{1-7}) in Theorem (\ref{T1-6}).
\end{remark}

\end{document}